\documentclass[11pt]{amsart}
\usepackage{amsfonts,amssymb}
\usepackage{amsmath}
\usepackage{amsthm,amscd,latexsym,mathrsfs,hyperref}

\newtheorem{theorem}{Theorem}[section]
\newtheorem{corollary}[theorem]{Corollary}
\newtheorem{lemma}[theorem]{Lemma}
\newtheorem{proposition}[theorem]{Proposition}

\theoremstyle{definition}
\newtheorem{definition}{Definition}[section]
\newtheorem{remark}{Remark}[section]

\newcommand{\B}{\mathcal{B}}

\DeclareMathOperator{\esssup}{ess\,sup}

\DeclareMathOperator{\hypo}{\mathrm{hypo}}

\DeclareMathOperator{\ran}{ran}
\DeclareMathOperator{\tr}{tr}

\newcommand\supp{\mathop{\rm supp}}

\begin{document}

\title[Analytic lifts of operator concave functions]{Analytic lifts of operator concave functions}
\author[Mikl\'os P\'alfia]{Mikl\'os P\'alfia}
\address{Department of Mathematics, Sungkyunkwan University, Suwon 440-746, Korea.}
\address{Bolyai Institute, Interdisciplinary Excellence Centre, University of Szeged, H-6720 Szeged, Aradi v\'ertan\'uk tere 1., Hungary.}
\email{palfia.miklos@aut.bme.hu}


\dedicatory{Dedicated to the memory of my grandfather B\'ela P\'alfia (1928-2020).}

\date{\today}

\begin{abstract}
The motivation behind this paper is threefold. Firstly, to study, characterize and realize operator concavity along with its applications to operator monotonicity of free functions on operator domains that are not assumed to be matrix convex. Secondly, to use the obtained Schur complement based representation formulas to analytically extend operator means of probability measures and to emphasize their study through random variables. Thirdly, to obtain these results in a decent generality. That is, for domains in arbitrary tensor product spaces of the form $\mathcal{A}\otimes\mathcal{B}(E)$, where $\mathcal{A}$ is a Banach space and $\mathcal{B}(E)$ denotes the bounded linear operators over a Hilbert space $E$. Our arguments also apply when $\mathcal{A}$ is merely a locally convex space.
\end{abstract}

\maketitle

\section{Introduction}
The main object of our investigations is a free function $F:D(E)\mapsto\mathcal{B}(E)$ with self-adjoint domain and range satisfying
\begin{equation}\label{eq0:concavity}
F\left((I_{\mathcal{A}}\otimes W^*)X(I_{\mathcal{A}}\otimes W)\right)\geq W^*F(X)W
\end{equation}
for each isometry $W:E\mapsto K$ for Hilbert spaces $E,K$ and $X\in D(K)$ such that also $W^*XW\in D(E)\subseteq\mathcal{A}\otimes\mathcal{B}(E)$, where $I_{\mathcal{A}}$ denotes the identity map of a Banach space $\mathcal{A}$. In particular if $(D(E))$ is matrix convex, then $W^*XW\in D(E)$ always when $X\in D(K)$, however this is not required in this paper. It is known by \cite{palfia1}, that \eqref{eq0:concavity} characterizes operator concave free functions determined by
\begin{equation}\label{eq0:concavity2}
F\left((1-\lambda)X+\lambda Y\right)\geq (1-\lambda)F\left(X\right)+\lambda F\left(Y\right)
\end{equation}
for $\lambda\in[0,1]$, $X,Y\in D(E)$ on a matrix convex domain $(D(E))$, so one might think of an $F$ above in \eqref{eq0:concavity} as a partially defined operator concave function. A conceptually similar problem in the particular case of power means of positive numbers is treated successfully in \cite{limpalfia2} and then in \cite{lawsonlim1} by lifting the real function into a fully non-commutative through characterizing nonlinear operator equations. The other paper known to the author that does not assume matrix convexity of the domain is the foundational material \cite{agler}, in which local monotonicity is characterized for a real multivariable function through construction of analytic extensions to upper half-planes. However \cite{agler} does not succeed in lifting up the real function into a full non-commutative free function on an enclosing matrix convex domain, nor can it show that it preserves the partial order. In this paper this problem of non-commutative lifts is eliminated by transforming the problem in section 2 into a more suitable form handled in Theorem~\ref{T:Loewner_several_var2} for real multivariable functions that preserves the partial order between commuting tuples of matrices. Theorem~\ref{T:Loewner_several_var2} constructs full non-commutative analytic lifts of these real multivariable functions to matrix convex hulls of the original domain. This is based on the more general Corollary~\ref{C:OperatorModel} characterizing functions satisfying \eqref{eq0:concavity}.

Usually matrix convexity of the domain $(D(E))$ is essential for the machinery of various further existing results characterizing operator monotone or operator concave \eqref{eq0:concavity2} functions, like \cite{helton,palfia1} and \cite{pascoe} when $\mathcal{A}=\mathbb{C}^k$ for a positive integer $k$, and in \cite{gaalpalfia,pascoe1,pascoe2} when $\mathcal{A}$ is an operator system. Excluding \cite{palfia1}, the other existing results in the field are restricted to the case when $\dim(E)<\infty$ and $F$ is continuous with respect to finitely open topologies used to study holomorphic functions in general, see the monograph \cite{verbovetskyi}. This essentially renders investigations in \cite{gaalpalfia,pascoe1,pascoe2} in the norm topology restricted to matrix convex matricial domains of $M_n(\mathcal{A})\simeq M_n(\mathbb{C})\otimes \mathcal{A}$ for an operator space $\mathcal{A}$. In this paper $\mathcal{A}$ can be any Banach space and $E$ can be infinite dimensional over the same ground field. Actually all results of section 3 can be worked out in exactly the same way for a locally convex vector space $\mathcal{A}$ as remarked there. However we will not need that here, for our applications in the last section, the Banach space case suffices.

The above mentioned restrictions of the state of the art of free function theory become apparent if we consider the recent developments in the theory of operator means of probability measures of positive operators in \cite{hiailawsonlim,hiailim,limpalfia,palfia2}. There, one studies functions $F:\mathcal{P}^\infty(\mathbb{P}(E))\mapsto\mathbb{P}(E)$ on the cone of probability measures $\mathcal{P}^\infty(\mathbb{P}(E))$ over the positive invertible operators $\mathbb{P}(E)$, which preserve the \emph{stochastic order} \cite{hiailawsonlim} of probability measures. We show that one can lift such a function $F$ into an operator monotone free function $\hat{F}:L^\infty([0,1],\lambda)^+\otimes \mathbb{P}(E)\mapsto\mathbb{P}(E)$ of $\mathbb{P}(E)$-valued random variables, thus satisfies \eqref{eq0:concavity}. Then we apply our results to this setting, to analytically continue an operator mean in several variables to the probability measure setting, thus obtaining $\hat{F}:L^\infty([0,1],\lambda)^+\otimes \mathbb{P}(E)\mapsto\mathbb{P}(E)$. This provides a realization for a class of operator means whose study were initiated in \cite{palfia2} and put into a framework in \cite{hiailim}. The main results in this topic are in section 4, specifically Theorem~\ref{T:OpMeanProbMeas}, Corollary~\ref{C:OpMeansProbMeasure} and Definition~\ref{D:directsumProbMeas}

The results of the paper are self-contained in the sense, that we essentially use only standard operator theoretic results available for example in \cite{takesaki} to study free functions. Detailed structure theory of free functions like the monograph \cite{verbovetskyi} is not required. In section 4 we also make use of the theory of the stochastic order of probability measures and some related results which are from probability theory.

In the following, we explicitly introduce the basic definitions of the objects to be studied in this paper. All vector spaces are over the ground fields $\mathbb{R}$ or $\mathbb{C}$ respectively. Let $\mathcal{A}$ be a vector space and let $I_\mathcal{A}:\mathcal{A}\mapsto\mathcal{A}$ denote the identity map.
\begin{definition}[Free set and matrix convex set]\label{D:matrix_covexity}
A collection $(D(E))$ of sets of operators $D(E)\subseteq\mathcal{A}\otimes\mathcal{B}(E)$ for each Hilbert space $E$ over the ground field $\mathbb{R}$ or $\mathbb{C}$ is a called a \emph{free set} whenever for all Hilbert spaces $E,K$ we have the following:
\begin{itemize}
	\item[1)] $(I_\mathcal{A}\otimes U^*)D(E)(I_\mathcal{A}\otimes U)\subseteq D(K)$ for all unitary $U:K\mapsto E$.
	\item[2)] $D(E)\oplus D(K)\subseteq D(E\oplus K)$.
\end{itemize}

If additionally (2) holds for any linear isometry $U:K\mapsto E$, then $(D(E))$ is a \emph{matrix convex set}. 

Sometimes the collection $(D(E))$ will be restricted to the case $\dim(E)<\infty$. In that case, for all other involved Hilbert spaces $K$ we assume $\dim(K)<\infty$ as well.
\end{definition}

We remark that if a given free set $(D(E))$ is matrix convex, then according to \cite{helton4} each $D(E)$ is convex in the usual sense.

\begin{definition}[Free function]\label{D:freeFunction}
Let $\mathcal{L}$ be a fixed Hilbert space. A collection of functions $F:D(E)\mapsto \mathcal{B}(\mathcal{L}\otimes E)$ indexed by $E$ for a free set $D(E)\subseteq\mathcal{A}\otimes\mathcal{B}(E)$ defined for all Hilbert spaces $E,K$ is called a \textit{free function} whenever for all $A \in D(E)$ and $B\in D(K)$, we have
\begin{itemize}
	\item[1)] unitary invariance, that is
	$$F((I_\mathcal{A}\otimes U^*)A(I_\mathcal{A}\otimes U))=(I_\mathcal{L}\otimes U^{*})F(A)(I_\mathcal{L}\otimes U)$$ 
	holds for all unitaries $U:E\mapsto K$;
	\item[2)] direct sum invariance, that is
	$$F\left(A \oplus B\right)=F(A) \oplus F(B).$$
\end{itemize}
\end{definition}
In the paper we assume that $\mathcal{L}=\mathbb{C}$, since given a free function $F:D(E)\mapsto \mathcal{B}(\mathcal{L}\otimes E)$, one can study its slices $l(F):D(E)\mapsto \mathcal{B}(E)$ instead, where $l\in\mathcal{B}(\mathcal{L})^*_+$ is a state of $\mathcal{B}(\mathcal{L})$, since $l(F)$ is then also a free function in the same class as $F$ itself regarding operator concavity or monotonicity, so essentially the same techniques apply to them.


\section{Lifted hypographs are matrix convex}
We use $\mathbb{P}(E)$ to denote the cone of invertible positive and $\mathbb{S}(E)\supset \mathbb{P}(E)$ the self-adjoint bounded linear operators over the Hilbert space $E$, so that $\mathbb{P}(\mathbb{C})$ denotes the positive reals.

\begin{definition}[cf. \cite{agler}]
A real function $f:\mathbb{P}(\mathbb{C})^k\mapsto\mathbb{P}(\mathbb{C})$ is said to be \emph{globally operator monotone}, if for any $X\leq Y\in\mathbb{CP}(E)^k$, $\dim(E)<+\infty$ we have $f(X)\leq f(Y)$, where $\mathbb{CP}(E)^k$ denotes the set of pairwise commuting $k$-tuples of invertible positive bounded linear operators on $E$, and $f(X):=U^*f(\Lambda)U$ where $X=U^*\Lambda U$ denotes the joint spectral decomposition of the pairwise commuting tuple $X$ and $f(\Lambda):=\bigoplus_{i=1}^k f(\{\Lambda_1\}_{ii},\ldots,\{\Lambda_k\}_{ii})$.
\end{definition}




\begin{proposition}\label{P:monotone_concave}
Let $f:\mathbb{P}(\mathbb{C})^k\mapsto\mathbb{P}(\mathbb{C})$ be a globally operator monotone function. Then for any isometry $W:E\mapsto K$ between finite dimensional Hilbert spaces $E,K$ and any $X\in\mathbb{CP}(K)^k$ such that $W^*XW\in\mathbb{CP}(E)^k$ we have
\begin{equation}\label{eq0:P:monotone_concave}
W^*f(X)W\leq f(W^*XW).
\end{equation}
In particular $f$ is concave and continuous as a real function.
\end{proposition}
\begin{proof}
Let
\begin{equation*}
U:=\left[ \begin{array}{cc}
W & (I-WW^*)^{1/2} \\
(I-W^*W)^{1/2} & -W^* \end{array} \right]=
\left[ \begin{array}{cc}
W & (I-WW^*)^{1/2} \\
0 & -W^* \end{array} \right]
\end{equation*}
denote a unitary dilation of the isometry $W$, i.e. $U^*U=UU^*=I$ on $E\oplus K$. Now choose arbitrary $A\in\mathbb{CP}(E)^k$ and let
\begin{equation*}
C:=(I-WW^*)^{1/2}X(I-WW^*)^{1/2}+WAW^*.
\end{equation*}
Then we have
\begin{equation*}
U^*\left[ \begin{array}{cc}
X & 0 \\
0 & A \end{array} \right]U=
\left[ \begin{array}{cc}
W^*XW & W^*X(I-WW^*)^{1/2} \\
(I-WW^*)^{1/2}XW & C \end{array} \right].
\end{equation*}
Set $D:=-W^*X(I-WW^*)^{1/2}$ and notice that for any given $\epsilon>0$
\begin{equation*}
\left[ \begin{array}{cc}
W^*XW+\epsilon I & 0 \\
0 & 2zI \end{array} \right]-U^*\left[ \begin{array}{cc}
X & 0 \\
0 & A \end{array} \right]U\geq
\left[ \begin{array}{cc}
\epsilon I & D \\
D^* & zI \end{array} \right]
\end{equation*}
if $zI \geq C=(I-WW^*)^{1/2}X(I-WW^*)^{1/2}+WAW^*$ for $z\in\mathbb{P}(\mathbb{C})^k$. The last $k$-tuple of block matrices above is positive semi-definite if $z_iI \geq\frac{1}{\epsilon} D_iD_i^*$ for all $1\leq i\leq k$. So, for sufficiently large positive $k$-tuple $z$ we have
\begin{equation*}
U^*\left[ \begin{array}{cc}
X & 0 \\
0 & A \end{array} \right]U\leq \left[ \begin{array}{cc}
W^*XW+\epsilon I & 0 \\
0 & 2zI \end{array} \right].
\end{equation*}
For such $z>0$, by the global operator monotonicity of $f$ we get
\begin{equation*}
f\left(U^*\left[ \begin{array}{cc}
X & 0 \\
0 & A \end{array} \right]U\right)\leq \left[ \begin{array}{cc}
f(W^*XW+\epsilon I) & 0 \\
0 & f(2z)I \end{array} \right].
\end{equation*}
We also have that
\begin{equation*}
\begin{split}
&f\left(U^*\left[ \begin{array}{cc}
X & 0 \\
0 & A \end{array} \right]U\right)=
U^*\left[ \begin{array}{cc}
f(X) & 0 \\
0 & f(A) \end{array} \right]U\\
&=\left[ \begin{array}{cc}
W^*f(X)W & W^*f(X)(I-WW^*)^{1/2} \\
(I-WW^*)^{1/2}f(X)W & \begin{array}{c} (I-WW^*)^{1/2}f(X)(I-WW^*)^{1/2}+ \\
+Wf(A)W^*\end{array} \end{array} \right],
\end{split}
\end{equation*}
hence we obtain that
\begin{equation}\label{eq:P:monotone_concave}
W^*f(X)W\leq f(W^*XW+\epsilon I).
\end{equation}
Now since $f$ is monotone, $f(X+\epsilon{I})$ for $\epsilon>0$ forms a decreasing net bounded from below by $f(X)$, thus the right limit
$$f^{+}(X):=\inf_{\epsilon>0}f(X+\epsilon{I})=\lim_{\epsilon\to 0+}f(X+\epsilon{I})$$
exists for all $X\in\mathbb{CP}(K)^k$ and $f^{+}$ is a multivariable real function. Hence for any $\epsilon>0$, using \eqref{eq:P:monotone_concave} with $W=[\lambda^{1/2}I, (1-\lambda)^{1/2}I]$ where $\lambda\in[0,1]$ and $X=\left[ \begin{array}{cc}
a & 0 \\
0 & b \end{array} \right]$ with $a,b\in\mathbb{CP}(\mathbb{C})^k$, we obtain
\begin{equation*}
\lambda f^{+}(a)+(1-\lambda)f^{+}(b)\leq \lambda f(a+\epsilon{I})+(1-\lambda)f(b+\epsilon{I})\leq f(\lambda a+(1-\lambda)b+2\epsilon{I}).
\end{equation*}
Taking the limit $\epsilon\to 0+$ we obtain that
\begin{equation*}
\lambda f^{+}(a)+(1-\lambda)f^{+}(b)\leq f^{+}(\lambda a+(1-\lambda)b),
\end{equation*}
i.e. the real function $f^{+}$ is concave, thus continuous. Also
\begin{equation*}
f(X)\leq f^{+}(X)\leq f(X+\epsilon{I})
\end{equation*}
for all $\epsilon>0$. Since $f$ is monotone increasing, we have
\begin{equation*}
f^{+}({X}-\epsilon{I})\leq f({X})\leq f^{+}({X}),
\end{equation*}
and since $f^{+}$ is continuous we get that $f=f^{+}$ by taking the limit $\epsilon\to 0+$. Hence we can also take the limit $\epsilon\to 0+$ in \eqref{eq:P:monotone_concave} proving \eqref{eq0:P:monotone_concave}.
\end{proof}

\begin{corollary}\label{C:monotoneJensen}
Under the assumptions of Proposition~\ref{P:monotone_concave}, \eqref{eq0:P:monotone_concave} remains true with contractions $W:E\mapsto K$, that is $\|W\|\leq 1$.
\end{corollary}
\begin{proof}
If $\|W\|\leq 1$, then $(I-W^*W)^{1/2}$ is not necessarily $0$. However the block operator matrix $U$ is still unitary and we can choose $A=0$ in the proof with $f(0):=0$ since $f|_{\mathbb{P}(\mathbb{C})^k}>0$, and the same block operator matrix argumentation goes through, leading to \eqref{eq0:P:monotone_concave}.
\end{proof}

\begin{definition}[Matrix convex hull]\label{D:MatrixConvexHull}
Given a disjoint union of sets $(C(E))$ for each Hilbert space $\dim(E)<+\infty$, its matrix convex hull, denoted as $(\mathrm{co}^{\mathrm{mat}}C(E))$ for each Hilbert space $\dim(E)<+\infty$, is defined as the smallest matrix convex set containing $(C(E))$. By Proposition 2.6 in \cite{helton4}, it is known that if $(C(E))$ is closed under direct sums, then
$$\mathrm{co}^{\mathrm{mat}}C(E):=\{V^*XV:X\in C(K),\dim(K)<+\infty,V:E\mapsto K\text{ an isometry}\}.$$
\end{definition}

Notice that in general convex combinations are themselves matrix convex combinations, since $(1-\lambda)A+\lambda B=V^*(A\oplus B)V$, where $V=\left[ \begin{array}{c}
(1-\lambda) \\
\lambda \end{array} \right]$ is an isometry for $\lambda\in[0,1]$. Provided the above definition of the matrix convex hull, the following result is almost immediate.

\begin{lemma}\label{L:CoMatHullP}
We have $\mathrm{co}^{\mathrm{mat}}\mathbb{CP}^k(E)=\mathbb{P}^k(E)$ for each $\dim(E)<\infty$.
\end{lemma}
\begin{proof}
Let $A\in\mathbb{P}^k(E)$ with $\dim(E)=n$, so that $A_i\geq \epsilon I_E$ for a small enough $\epsilon>0$. In spectral decomposition form, we also have
\begin{equation*}
A=\left(\sum_{i=1}^na^1_iu^1_iu^{1*}_i,\ldots,\sum_{i=1}^na^k_iu^k_iu^{k*}_i\right)
\end{equation*}
where $u^i_l$ are eigenvectors and $a^i_l$ are the corresponding eigenvalues of $A_i$. By looking at the above form, it is clear that each $A_i$ is written as a finite convex combination of rank one matrices of the form $cuu^{*}+d I_E$ where $c,d\in\mathbb{P}(\mathbb{C})$ and $u\in E$, $\|u\|=1$. We can write $cuu^*=(e_1\otimes u^*)^*ce_1e_1^*(e_1\otimes u^*)$, an isometric inclusion of $ce_1e_1^*\in\mathbb{P}(\mathbb{C})$, so by extending $e_1\otimes u^*$ into a unitary $U$, we get that $cuu^{*}+d I_E=U^*((c+d)\oplus (\oplus_{j=2}^nd))U$, itself a matrix convex combination. Thus it follows that $A$ is actually a finite convex combination of elements where only one coordinate of the $k$-tuple is not necessarily equal to $z I_E$ for some $z\in\mathbb{P}(\mathbb{C})$, and such elements are also finite matrix convex combinations of elements of $\mathbb{P}^k(\mathbb{C})$.
\end{proof}

Now consider the hypograph
$$\hypo(f):=(\hypo(f)(K)):=(\{(Y,X)\in\mathbb{S}(K)\times\mathbb{CP}(K)^k:Y\leq f(X)\})$$
of a real function $f:\mathbb{P}(\mathbb{C})^k\mapsto\mathbb{P}(\mathbb{C})$ for $\dim(K)<+\infty$. We should think about the real function $f$ and its $\hypo(f)$ as a partially defined free function and its partially defined hypograph.

\begin{theorem}\label{T:convex_hull_saturated}
Let $f:\mathbb{P}(\mathbb{C})^k\mapsto\mathbb{P}(\mathbb{C})$ be a real function. Then $f$ is globally operator monotone if and only if for each $(Y,X)\in\mathrm{co}^{\mathrm{mat}}(\hypo(f))(E)$ with $\dim(E)<+\infty$ and $X\in\mathbb{CP}(E)^k$ we have that $Y\leq f(X)$.
\end{theorem}
\begin{proof}
Suppose first that $f$ is globally operator monotone. Let $(Y,X)\in\mathrm{co}^{\mathrm{mat}}(\hypo(f)(E))$ with $\dim(E)<+\infty$ and $X\in\mathbb{CP}(E)^k$. Then by the definition of the matrix convex hull there exists an isometry $W:E\mapsto K$ between the finite dimensional Hilbert spaces $E,K$ and a $(y,x)\in\mathbb{S}(K)\times\mathbb{CP}(K)^k$ with $y\leq f(x)$ such that $Y=W^*yW$ and $X=W^*xW$. Then it follows that $Y\leq W^*f(x)W$, so by Proposition~\ref{P:monotone_concave} we get that $W^*f(x)W\leq f(W^*xW)=f(X)$.

To see the converse implication, consider the function
\begin{equation*}
h_v(X):=\sup\{v^*Yv:(Y,X)\in\mathrm{co}^{\mathrm{mat}}(\hypo(f))(E)\}
\end{equation*}
for $v\in E$ and $X\in\mathbb{P}(E)^k$. Since $\mathrm{co}^{\mathrm{mat}}(\hypo(f))$ is matrix convex, we have that $\mathrm{co}^{\mathrm{mat}}(\hypo(f))(E)$ is a convex set, it follows that $h_v$ is a bounded from below concave, thus by Proposition 3.5.4 in \cite{niculescu}, norm-continuous real valued function. Moreover by the assumption if $X\in\mathbb{CP}(E)^k$ then for each $(Y,X)\in\mathrm{co}^{\mathrm{mat}}(\hypo(f))(E)$ we have that $Y\leq f(X)$, thus we must have $h_v(X)=v^*f(X)v$. It is also clear by the definition of $\mathrm{co}^{\mathrm{mat}}(\hypo(f))(E)$ that $h_v\geq 0$ on its domain which is the whole $\mathbb{P}(E)^k$ by Lemma~\ref{L:CoMatHullP}. Now assume that $A,B\in\mathbb{CP}(E)^k$ and $A<B$. Let $t\in(0,1)$. Then we have that
\begin{equation*}
tB=tA+(1-t)\left[\frac{t}{1-t}(B-A)\right]
\end{equation*}
where $\frac{t}{1-t}(B-A)\in\mathbb{P}(E)^k$. Thus the concavity and positivity of $h_v$ yields
\begin{equation*}
h_v(tB)\geq th_v(A)+(1-t)h_v\left(\frac{t}{1-t}(B-A)\right)\geq th_v(A),
\end{equation*}
so letting $t\to 1-$ in the above implies $h_v(B)\geq h_v(A)$. Then again using the continuity of $h_v$ we obtain $h_v(B)\geq h_v(A)$ as well, when we have $B\geq A$. From this, since $v\in E$ was arbitrary, we obtain $f(B)\geq f(A)$ as desired.
\end{proof}

\section{Free analytic lifts through models}
Let $\mathcal{A}$ denote a Banach space in this section. 
All tensor products in the subsequent sections are understood to be projective, see chapter IV.2. in \cite{takesaki} for more information, 
however this particular choice of cross-norm does not make an essential difference in the calculations. For a vector space $\mathcal{V}$ the map $I_\mathcal{V}$ is understood to be the identity homomorphism. The first result characterizes concavity through isometric conjugations. Such maps are also called Jensen-type maps if they satisfy a similar reversed inequality \cite{hansenMoslehianNajafi}. The characterizing inequality \eqref{eq:P:ConcaveIsom} will play a key role in this section.

\begin{proposition}\label{P:ConcaveIsom}
Let $(D(E))$ with $D(E)\subseteq \mathcal{A}\otimes \mathcal{B}(E)$ denote a self-adjoint matrix convex set and let $F:D(E)\mapsto \mathcal{B}(E)$ be a free function. Then $F$ is operator concave if and only if for each isometry $W:E\mapsto K$ and $X\in D(K)$ we have 
\begin{equation}\label{eq:P:ConcaveIsom}
F\left((I_{\mathcal{A}}\otimes W^*)X(I_{\mathcal{A}}\otimes W)\right)\geq W^*F(X)W.
\end{equation}
\end{proposition}
\begin{proof}
$(\Rightarrow):$ Let
\begin{equation*}
U:=\left[ \begin{array}{cc}
W & (I-WW^*)^{1/2} \\
(I-W^*W)^{1/2} & -W^* \end{array} \right]=
\left[ \begin{array}{cc}
W & (I-WW^*)^{1/2} \\
0 & -W^* \end{array} \right]
\end{equation*}
denote the unitary dilation of the isometry $W$, i.e. $U^*U=UU^*=I$ on $E\oplus K$. Now choose arbitrary $A\in D(E)$ and let
\begin{equation*}
C:=(I_{\mathcal{A}}\otimes(I-WW^*)^{1/2})X(I_{\mathcal{A}}\otimes(I-WW^*)^{1/2})+(I_{\mathcal{A}}\otimes W)A(I_{\mathcal{A}}\otimes W^*).
\end{equation*}
Then we have
\begin{equation*}
\begin{split}
&(I_{\mathcal{A}}\otimes U^*)\left[ \begin{array}{cc}
X & 0 \\
0 & A \end{array} \right](I_{\mathcal{A}}\otimes U)\\
&=\left[ \begin{array}{cc}
(I_{\mathcal{A}}\otimes W^*)X(I_{\mathcal{A}}\otimes W) & (I_{\mathcal{A}}\otimes W^*)X(I_{\mathcal{A}}\otimes(I-WW^*)^{1/2}) \\
(I_{\mathcal{A}}\otimes(I-WW^*)^{1/2})X(I_{\mathcal{A}}\otimes W) & C \end{array} \right].
\end{split}
\end{equation*}
Also notice that
\begin{equation*}
\begin{split}
\frac{1}{2}(I_{\mathcal{A}}\otimes U^*)&\left[ \begin{array}{cc}
X & 0 \\
0 & A \end{array} \right](I_{\mathcal{A}}\otimes U)+\frac{1}{2}\left[ \begin{array}{cc}
I & 0 \\
0 & -I \end{array} \right](I_{\mathcal{A}}\otimes U^*)\left[ \begin{array}{cc}
X & 0 \\
0 & A \end{array} \right]\\
&\times(I_{\mathcal{A}}\otimes U)\left[ \begin{array}{cc}
I & 0 \\
0 & -I \end{array} \right]=\left[ \begin{array}{cc}
(I_{\mathcal{A}}\otimes W^*)X(I_{\mathcal{A}}\otimes W) & 0 \\
0 & C \end{array} \right].
\end{split}
\end{equation*}
Then we have
\begin{equation*}
\begin{split}
&\left[ \begin{array}{cc}
F((I_{\mathcal{A}}\otimes W^*)X(I_{\mathcal{A}}\otimes W)) & 0 \\
0 & F(C) \end{array} \right]\\
&=F\left(\left[ \begin{array}{cc}
(I_{\mathcal{A}}\otimes W^*)X(I_{\mathcal{A}}\otimes W) & 0 \\
0 & C \end{array} \right]\right)\\
&=F\left(\frac{1}{2}(I_{\mathcal{A}}\otimes U^*)\left[ \begin{array}{cc}
X & 0 \\
0 & A \end{array} \right](I_{\mathcal{A}}\otimes U)\right.\\
&\quad\quad \left.+\frac{1}{2}\left[ \begin{array}{cc}
I & 0 \\
0 & -I \end{array} \right](I_{\mathcal{A}}\otimes U^*)\left[ \begin{array}{cc}
X & 0 \\
0 & A \end{array} \right](I_{\mathcal{A}}\otimes U)\left[ \begin{array}{cc}
I & 0 \\
0 & -I \end{array} \right]\right)\\
&\geq \frac{1}{2}F\left((I_{\mathcal{A}}\otimes U^*)\left[ \begin{array}{cc}
X & 0 \\
0 & A \end{array} \right](I_{\mathcal{A}}\otimes U)\right)\\
&\quad+\frac{1}{2}F\left(\left[ \begin{array}{cc}
I & 0 \\
0 & -I \end{array} \right](I_{\mathcal{A}}\otimes U^*)\left[ \begin{array}{cc}
X & 0 \\
0 & A \end{array} \right](I_{\mathcal{A}}\otimes U)\left[ \begin{array}{cc}
I & 0 \\
0 & -I \end{array} \right]\right)\\
&=\frac{1}{2}(I_{\mathcal{A}}\otimes U^*)\left[ \begin{array}{cc}
F(X) & 0 \\
0 & F(A) \end{array} \right](I_{\mathcal{A}}\otimes U)\\
&\quad+\frac{1}{2}\left[ \begin{array}{cc}
I & 0 \\
0 & -I \end{array} \right](I_{\mathcal{A}}\otimes U^*)\left[ \begin{array}{cc}
F(X) & 0 \\
0 & F(A) \end{array} \right](I_{\mathcal{A}}\otimes U)\left[ \begin{array}{cc}
I & 0 \\
0 & -I \end{array} \right]\\
&=\left[ \begin{array}{cc}
(I_{\mathcal{A}}\otimes W^*)F(X)(I_{\mathcal{A}}\otimes W) & 0 \\
0 & \begin{array}{c} (I_{\mathcal{A}}\otimes(I-WW^*)^{1/2})F(X) \\
 \times(I_{\mathcal{A}}\otimes(I-WW^*)^{1/2}) \\
 +(I_{\mathcal{A}}\otimes W)F(A)(I_{\mathcal{A}}\otimes W^*) \end{array} \end{array} \right].
\end{split}
\end{equation*}
Thus \eqref{eq:P:ConcaveIsom} follows.

$(\Leftarrow):$ For $t\in[0,1]$ let $W=\left[ \begin{array}{c}
(1-t)^{1/2}I_E \\
t^{1/2}I_E \end{array} \right]$ so that $W^*W=I_E$, an isometry. Let $X,Y\in D(E)$. Then by \eqref{eq:P:ConcaveIsom} we have
\begin{equation*}
\begin{split}
&F((1-t)X+tY)=F\left((I_{\mathcal{A}}\otimes W^*)\left[ \begin{array}{cc}
X & 0 \\
0 & Y \end{array} \right](I_{\mathcal{A}}\otimes W)\right)\\
&\geq (I_{\mathcal{A}}\otimes W^*)F\left(\left[ \begin{array}{cc}
X & 0 \\
0 & Y \end{array} \right]\right)(I_{\mathcal{A}}\otimes W)\\
&=(I_{\mathcal{A}}\otimes W^*)\left[ \begin{array}{cc}
F(X) & 0 \\
0 & F(Y) \end{array} \right](I_{\mathcal{A}}\otimes W)\\
&=(1-t)F(X)+tF(Y).
\end{split}
\end{equation*}
\end{proof}

\begin{corollary}\label{C:ConcaveContr}
Under the assumptions of Proposition~\ref{P:ConcaveIsom} if also $0\in D(\mathbb{C})$ and $F(0)\geq 0$, then the equivalence in Proposition~\ref{P:ConcaveIsom} remains true with contractions $W:E\mapsto K$ in \eqref{eq:P:ConcaveIsom}, that is $\|W\|\leq 1$.
\end{corollary}
\begin{proof}
Only the $(\Rightarrow)$ implication in Proposition~\ref{P:ConcaveIsom} requires further consideration, since if $\|W\|\leq 1$ only, then $(I-W^*W)^{1/2}$ is not necessarily $0$. However the block operator matrix $U$ is still unitary and since $0\in D(\mathbb{C})$, we can choose $A=0$ in the proof of $(\Rightarrow)$ and the same block operator matrix argumentation goes through leading to \eqref{eq:P:ConcaveIsom}.
\end{proof}

The following lemma has been proved and used by a number of authors before. For its proof we refer to \cite{effros,palfia1}.

\begin{lemma}[Lemma 3.6. \cite{palfia1}]\label{L:existTop}
Suppose $\mathcal{F}$ is a convex set of weak-$*$ continuous affine linear mappings $f:\mathcal{B}^+_1(E)^{*}\mapsto \mathbb{R}$ with respect to a duality. If for each $f\in\mathcal{F}$ there exists a $T\in\mathcal{B}^+_1(E)^{*}$ such that $f(T)\geq 0$, then there exists a $\mathcal{T}\in\mathcal{B}^+_1(E)^{*}$ such that $f(\mathcal{T})\geq 0$ for every $f\in\mathcal{F}$.
\end{lemma}

\begin{lemma}\label{L:existT}
Let $D=(D(E))$ be a matrix convex set, where $D(E)\subseteq \mathcal{A}\otimes \mathcal{B}(E)$ and $0\in D(\mathbb{C})$. Let a linear functional $\Lambda:\mathcal{A}\otimes \mathcal{B}(N)\mapsto\mathbb{R}$ be given for a fixed $N$. If $\Lambda(X)\leq 1$ for each $X\in D(N)$, then there exists a $T\in\mathcal{B}^+_1(N)^{*}$ 
such that for each Hilbert space $E$, and each $Y\in D(E)$ and each contraction $V:N\mapsto E$ we have
$$\Lambda((I_{\mathcal{A}}\otimes V^*)Y(I_{\mathcal{A}}\otimes V))\leq T(V^*V).$$
\end{lemma}
\begin{proof}
For a Hilbert space $K$, a point $Y\in D(K)$ and a $V:N\mapsto K$ contraction, define $f_{Y,V}:\mathcal{B}^+_1(N)^{*}\mapsto\mathbb{R}$ by
\begin{equation*}
f_{Y,V}(T):=T(V^*V)-\Lambda((I_{\mathcal{A}}\otimes V^*)Y(I_{\mathcal{A}}\otimes V)).
\end{equation*}
We claim that the collection $\mathcal{F}:=\{f_{Y,V}:Y,V\}$ is a convex set. Let $\lambda_i\geq 0$ for $1\leq i\leq n$ for a fixed integer $n$ and let $\sum_{i=1}^n\lambda_i=1$. Also let $(Y_i,V_i)$ be given where $Y_i\in D(K_i)$ for a Hilbert space $K_i$ and $V_i:N\mapsto K_i$ be a contraction for each $1\leq i\leq n$. Let $Z:=\oplus_{i=1}^nY_i$ and let $F$ denote the column operator matrix with entries $\sqrt{\lambda_i}V_i$. Then $Z\in D(\oplus K_i)$ and
\begin{equation*}
F^*F=\sum_{i=1}^n\lambda_iV_i^*V_i\leq \sum_{i=1}^n\lambda_iI=I.
\end{equation*}
By definition
\begin{equation*}
\sum_{i=1}^n\lambda_i(I_{\mathcal{A}}\otimes V_i^*)Y_i(I_{\mathcal{A}}\otimes V_i)=(I_{\mathcal{A}}\otimes F^*)Z(I_{\mathcal{A}}\otimes F)
\end{equation*}
and
\begin{equation*}
\sum_{i=1}^n\lambda_iT(V_i^*V_i)=T(F^*F)
\end{equation*}
for $T\in\mathcal{B}^+_1(N)^{*}$. Hence
\begin{equation*}
\sum_{i=1}^n\lambda_if_{Y_i,V_i}(T)=f_{Z,F}(T).
\end{equation*}

If $V$ has operator norm 1, by Proposition II.6.3.3. in \cite{blackadar} we can choose a norming state $\gamma\in\mathcal{B}^+_1(N)^{*}$ so that
\begin{equation*}
1=\|V\|^2=\gamma(V^*V).
\end{equation*}
Then for $T=\gamma$ it follows that
\begin{equation*}
f_{Y,V}(T)=T(V^*V)-\Lambda((I_{\mathcal{A}}\otimes V^*)Y(I_{\mathcal{A}}\otimes V))=1-\Lambda((I_{\mathcal{A}}\otimes V^*)Y(I_{\mathcal{A}}\otimes V)).
\end{equation*}
Since $(I_{\mathcal{A}}\otimes V^*)Y(I_{\mathcal{A}}\otimes V)\in D(N)$, the right hand side above is nonnegative. If the operator $V$ does not have norm one, we can rescale it to have norm 1 and follow the same argument to show that $f_{Y,V}(T)\geq 0$. So, for each $f_{Y,V}$ there exists a $T\in\mathcal{B}^+_1(N)^{*}$ such that $f_{Y,V}(T)\geq 0$, moreover each $f_{Y,V}$ is weak-$*$ continuous. Thus, by Lemma~\ref{L:existTop} there exists a $\mathcal{T}\in\mathcal{B}^+_1(N)^{*}$ such that $f_{Y,V}(\mathcal{T})\geq 0$ for every $Y$ and $V$.
\end{proof}



Similarly to the finite dimensional case of Definition~\ref{D:MatrixConvexHull}, given a disjoint union of sets $(C(E)\subseteq \mathcal{A}\otimes \mathcal{B}(E))$ for each Hilbert space $E$ closed under direct sums, its matrix convex hull is given as
$$\mathrm{co}^{\mathrm{mat}}C(E):=\bigcup_{K\text{ a Hilbert space}}\{V^*XV:X\in C(K),V:E\mapsto K\text{ an isometry}\}.$$
If $0\in C(\mathbb{C})$ then we also have
$$\mathrm{co}^{\mathrm{mat}}C(E)=\bigcup_{K\text{ a Hilbert space}}\{V^*XV:X\in C(K),V:E\mapsto K, \|V\|\leq 1\}.$$

Given a collection of sets $(D(E)\subseteq \mathcal{A}\otimes \mathcal{B}(E))$ closed under direct sums and a collection of functions $F:D(E)\mapsto\mathcal{B}(E)$ preserving direct sums, we consider its hypograph
$$\hypo(F):=(\hypo(F)(E)):=(\{(Y,X)\in\mathcal{B}(E)\times D(E):Y\leq f(X)\}).$$

\begin{proposition}\label{P:partialConcaveHypoGr}
Let a collection of self-adjoint sets $(D(E)\subseteq \mathcal{A}\otimes \mathcal{B}(E))$ closed under direct sums and a collection of functions $F:D(E)\mapsto\mathcal{B}(E)$ preserving direct sums be given. Then for each isometry $W:E\mapsto K$ and $X\in D(K)$ such that $(I_{\mathcal{A}}\otimes W^*)X(I_{\mathcal{A}}\otimes W)\in D(E)$ we have that
\begin{equation}\label{eq:P:partialConcaveHypoGr}
F\left((I_{\mathcal{A}}\otimes W^*)X(I_{\mathcal{A}}\otimes W)\right)\geq W^*F(X)W,
\end{equation}
if and only if for each $(Y,X)\in\mathrm{co}^{\mathrm{mat}}(\hypo(F))(E)$ with $X\in D(E)$ we have that $Y\leq F(X)$.

Moreover if $0\in D(\mathbb{C})$ and $F(0)\geq 0$ then the statement holds with contractions $W:E\mapsto K$ in \eqref{eq:P:partialConcaveHypoGr}.
\end{proposition}
\begin{proof}
A straightforward argumentation similar to the proof of Theorem~\ref{T:convex_hull_saturated} based on the expression of the matrix convex hull.
\end{proof}

\begin{remark}
For unitary $W$ we have that $W^{-1}$ is also an isometry, thus using \eqref{eq:P:partialConcaveHypoGr} twice, we can see that in this special case \eqref{eq:P:partialConcaveHypoGr} holds with equality. In \cite{hansenMoslehianNajafi} it was shown that if a map satisfies \eqref{eq:P:partialConcaveHypoGr} for all contractions $W$ and convexity (thus called a map of Jensen-type, see Definition 1.1. \cite{hansenMoslehianNajafi}), then it preserves direct sums since the proof of Lemma 3.1. (i) in \cite{hansenMoslehianNajafi} goes through. Their proof carries over to our case with minor modifications if the domain $(D(E))$ is matrix convex, thus in this case the direct sum invariance of $F$ in Proposition~\ref{P:partialConcaveHypoGr} can be dropped.
\end{remark}



\begin{proposition}\label{P:separating_pencil}
Let $(D(E))\ni 0$ and $F$ be as in Proposition~\ref{P:partialConcaveHypoGr} with $F|_D>0$. Assume that $\mathrm{co}^{\mathrm{mat}}(D)(E)$ has nonempty interior for each $E$. Let $N$ be a Hilbert space. Then for each interior point $A\in D(N)$ and each unit vector $v\in N$ there exists a completely bounded affine linear map $L_{F,A,v}:(\mathcal{B}(E),\mathcal{A}\otimes\mathcal{B}(E))\mapsto \mathcal{B}(N)^*\otimes\mathcal{B}(E)$ given as
\begin{equation*}
L_{F,A,v}(Y,X):=T(F,A,v)\otimes I_E-vv^*\otimes Y+\Lambda_{F,A,v}(X),
\end{equation*}
where $0\leq T(F,A,v)\in\mathcal{B}(N)^{*}$ and $\Lambda_{F,A,v}:\mathcal{A}\mapsto \mathcal{B}(N)^*$ is a self-adjoint completely bounded linear map, such that
\begin{itemize}
\item[(a)]$T(F,A,v)(I_N)=v^*F(A)v-\Lambda_{F,A,v}(A)$ and there exists $\epsilon>0$ such that $(1+\epsilon)A\in \mathrm{co}^{\mathrm{mat}}(D)(N)$ and $-\Lambda_{F,A,v}(A)\leq\frac{v^*F(A)v-v^*F((1+\epsilon)A)v}{\epsilon}$;
\item[(b)]For all $(Y,X)\in\hypo(F)$ we have $L_{F,A,v}(Y,X)\geq 0$;
\item[(c)]$\gamma^*L_{F,A,v}(F(A),A)\gamma=0$ where $\gamma=I_{N}$;
\item[(d)]For every $X$ in the interior of $\mathrm{co}^{\mathrm{mat}}(D)(E)$ there exists an $\epsilon>0$ such that $\left\langle V,L_{F,A,v}(0,X)V\right\rangle\geq \epsilon T(F,A,v)(V^*V)$.
\end{itemize}
\end{proposition}
\begin{proof}
Define the real valued function $h_v:\mathrm{co}^{\mathrm{mat}}(D)(N)\mapsto \mathbb{R}$ as $h_v(X):=\sup\{v^*Yv:(Y,X)\in\mathrm{co}^{\mathrm{mat}}(\hypo(F)(N))\}$. Since $\mathrm{co}^{\mathrm{mat}}(\hypo(F))$ is matrix convex, we have that $\mathrm{co}^{\mathrm{mat}}(\hypo(F))(E)$ is a convex set, also $F|_D\geq 0$, so it follows that $h_v$ is a bounded from below concave function on the real Banach-space of the self-adjoint part of $\mathcal{A}\otimes \mathcal{B}(N)$, thus norm-continuous by Proposition 3.5.4 in \cite{niculescu}. Moreover by Proposition~\ref{P:partialConcaveHypoGr} if $X\in D(N)$ then for each $(Y,X)\in\mathrm{co}^{\mathrm{mat}}(\hypo(F))(N)$ we have that $Y\leq F(X)$, thus we must have $h_v(X)=v^*F(X)v$ for all $X\in D(N)$.

It follows from the supporting hyperplane version of the Hahn-Banach theorem, more precisely Theorem 7.12 and 7.16 \cite{aliprantis}, that the norm-continuous convex function $g(X):=-h_v(X)$ has a subgradient at each interior point of its domain, thus at $A$. 
That is, there exists a self-adjoint continuous linear functional $\lambda\in (\mathcal{A}\otimes \mathcal{B}(N))^{*}$ such that
\begin{equation}\label{eq:P:separating_pencil1}
h_v(X)-h_v(A)\leq \lambda(X-A)
\end{equation}
for $X\in \mathrm{co}^{\mathrm{mat}}(D)(N)$ and for $X=A$ we have equality. Now let $(Y,X)\in\mathrm{co}^{\mathrm{mat}}(\hypo(F))(N)$. Then it follows from \eqref{eq:P:separating_pencil1} and the definition of $h_v$ that
\begin{equation}\label{eq:P:separating_pencil2}
v^*Yv-\lambda(X)\leq h_v(A)-\lambda(A).
\end{equation}

Notice that by the assumption we have $0\in D(\mathbb{C})$ and $F>0$, thus $0\in\mathrm{co}^{\mathrm{mat}}(\hypo(F))(E)$ for any Hilbert space $E$ and $h_v(A)-\lambda(A)>0$. Thus the linear functional $\Lambda(Y,X):=\frac{1}{h_v(A)-\lambda(A)}\left(v^*Yv-\lambda(X)\right)$ satisfies $\Lambda(Y,X)\leq 1$ for $(Y,X)\in\mathrm{co}^{\mathrm{mat}}(\hypo(F))(N)$, also $\mathrm{co}^{\mathrm{mat}}(\hypo(F))$ is a matrix convex set. Thus by Lemma~\ref{L:existT} there exists a $0\leq T\in\mathcal{B}^+_1(N)^{*}$ such that $T(I_N)=1$ and for any contraction $V:N\mapsto E$ we have
\begin{equation}\label{eq:P:separating_pencil3}
0\leq (h_v(A)-\lambda(A))T(V^*V)-v^*V^*YVv+\lambda((I_{\mathcal{A}}\otimes V^*)X(I_{\mathcal{A}}\otimes V))
\end{equation}
for any $(Y,X)\in\hypo(F)(E)$, moreover by \eqref{eq:P:separating_pencil1} choosing $V=I_{N}$ and $Y=F(A),X=A$ we get
\begin{equation}\label{eq:P:separating_pencil3.1}
0=(h_v(A)-\lambda(A))T(I_N)-v^*F(A)v+\lambda((I_{\mathcal{A}}\otimes I_N)A(I_{\mathcal{A}}\otimes I_N)).
\end{equation}

Now define
\begin{equation*}
T(F,A,v):=(h_v(A)-\lambda(A))T.
\end{equation*}
Next, by assumption $\mathrm{co}^{\mathrm{mat}}(D)(E)$ contains an open neighborhood of $0$ in $\mathcal{A}\otimes \mathcal{B}(E)$, thus there exists a $\hat{\rho}>0$ such that $B(0,\hat{\rho})\subseteq \mathrm{co}^{\mathrm{mat}}(D)(E)$. Thus by \eqref{eq:P:separating_pencil3} and norm continuity we have that
\begin{equation*}
0\leq T(F,A,v)(V^*V)+\lambda((I_{\mathcal{A}}\otimes V^*)X(I_{\mathcal{A}}\otimes V))
\end{equation*}
for $X\in\overline{B}(0,\hat{\rho})$ where $\overline{B}(0,\hat{\rho})$ denotes the norm closure of $B(0,\hat{\rho})$. Moreover $X\in\overline{B}(0,\hat{\rho})$ if and only if $-X\in\overline{B}(0,\hat{\rho})$, so we also have
\begin{equation*}
0\leq T(F,A,v)(V^*V)-\lambda((I_{\mathcal{A}}\otimes V^*)X(I_{\mathcal{A}}\otimes V)).
\end{equation*}
Then from the above it follows that
\begin{equation}\label{eq:P:separating_pencil4}
-T(F,A,v)(V^*V)\leq\lambda((I_{\mathcal{A}}\otimes V^*)X(I_{\mathcal{A}}\otimes V))\leq T(F,A,v)(V^*V)
\end{equation}
for $X\in\overline{B}(0,\hat{\rho})$. This together with Theorem IV.2.3. \cite{takesaki} ensure that the transpose map $\Lambda_{F,A,v}:\mathcal{A}\mapsto \mathcal{B}(N)^*$ of $\lambda\in(\mathcal{A}\otimes\mathcal{B}(N))^*$ is completely bounded and self-adjoint since $\lambda$ is a self-adjoint linear functional.

Consider the Hilbert space $\overline{\mathcal{B}(N,E)_{T(F,A,v)}}$ that we obtain by completing the quotient space $\mathcal{B}(N,E)/\{V\in\mathcal{B}(N,E):T(F,A,v)(V^*V)=0\}$ equipped with the positive definite Hermitian form
\begin{equation}\label{eq:P:separating_pencil5}
\left\langle W,V\right\rangle_{T(F,A,v)}:=T(F,A,v)(W^*V)
\end{equation}
for $W,V\in\mathcal{B}(N,E)$. 
 Then the right hand side of \eqref{eq:P:separating_pencil3} determines a quadratic form in $V\in\mathcal{B}(N,E)$, which gives rise to the densely defined symmetric linear operator
\begin{equation*}
\begin{split}
\left\langle W,L_{F,A,v}(Y,X)V\right\rangle_{T(F,A,v)}&:=T(F,A,v)(W^*V)-v^*W^*YVv\\
&\quad +\lambda((I_\mathcal{A}\otimes W^*)X(I_\mathcal{A}\otimes V))
\end{split}
\end{equation*}
for $V,W\in\mathcal{B}(N,E),Y\in\mathcal{B}(E)$ and $X\in\mathcal{A}\otimes \mathcal{B}(E)$. 
Then (b) and (c) of the assertion follows from \eqref{eq:P:separating_pencil3} and \eqref{eq:P:separating_pencil3.1} respectively and they also yield the first equality in (a). 

Furthermore inequality \eqref{eq:P:separating_pencil4} ensures that $L_{F,A,v}$ is a completely bounded affine linear map that admits the continuous linear extension $L_{F,A,v}(Y,X)$ which then is a bounded self-adjoint operator acting on $\overline{\mathcal{B}(N,E)_{T(F,A,v)}}$.

To see the remaining parts of (a), first we realize that by assumption $A$ is in the interior of $\mathrm{co}^{\mathrm{mat}}(D)(N)$, thus there exists an $\epsilon>0$ such that $X:=(1+\epsilon)A$ is in $\mathrm{co}^{\mathrm{mat}}(D)(N)$ as well. Then choosing $Y=F((1+\epsilon)A), X=(1+\epsilon)A$, from \eqref{eq:P:separating_pencil1} we calculate
\begin{equation*}
\begin{split}
0&\leq h_v((1+\epsilon)A)\leq h_v(A)-\lambda(A)+(1+\epsilon)\lambda(A)\\
\epsilon h_v(A)-\epsilon \lambda(A)&\leq (1+\epsilon)h_v(A)-h_v((1+\epsilon)A)\\
h_v(A)-\lambda(A)&\leq h_v(A)+\frac{h_v(A)-h_v((1+\epsilon)A)}{\epsilon},
\end{split}
\end{equation*}
thus the last inequality in (a) follows, since
\begin{equation*}
T(F,A,v)(I_N)=h_v(A)-\lambda(A).
\end{equation*}

Now to see (d), by assumption $V\in\overline{\mathcal{B}(N,E)_{T(F,A,v)}}$ and $X$ is in the interior of $\mathrm{co}^{\mathrm{mat}}(D)(E)$, thus there exists an $\epsilon>0$ such that $B(X,\epsilon)\subseteq \mathrm{co}^{\mathrm{mat}}(D)(E)$. Then there exists an $r>1$, such that $rX\in B(X,\epsilon)$. Let $c:=\lambda((I_{\mathcal{A}}\otimes V^*)X(I_{\mathcal{A}}\otimes V))$. Then by \eqref{eq:P:separating_pencil3} we have that
\begin{equation*}
0\leq\left\langle V,\left(L_{F,A,v}(0,rX\right)V\right\rangle=T(F,A,v)(V^*V)+rc,
\end{equation*}
thus
\begin{equation}\label{eq:P:separating_pencil6}
\begin{split}
0&<\left(1-\frac{1}{r}\right)T(F,A,v)(V^*V)\\
&\leq T(F,A,v)(V^*V)+c\\
&=T(F,A,v)(V^*V)+\lambda((I_{\mathcal{A}}\otimes V^*)X(I_{\mathcal{A}}\otimes V))\\
&=\left\langle V,L_{F,A,v}(0,X)V\right\rangle
\end{split}
\end{equation}
proving (d).
\end{proof}

\begin{remark}\label{R:LocallyConvexSpc}
In order to allow locally convex vector spaces $\mathcal{A}$ in Proposition~\ref{P:separating_pencil}, one needs to establish the continuity of $h_v$. This can be done using Proposition 4.4. in \cite{cobzas} which generalizes Proposition 3.5.4 in \cite{niculescu} to locally convex vector spaces under the same assumptions.
\end{remark}


Let $(D(E))\ni 0$ be as in Proposition~\ref{P:partialConcaveHypoGr}. Then for a Hilbert space $E$ and a dense set $E_0\in \{x\in E: \|x\|=1\}$ define the auxiliary vector space
\begin{equation*}
\mathcal{H}_{E,0}:=\bigoplus_{(X,v)\in (D(E),E_0)}E
\end{equation*}
and its completion $\mathcal{H}_{E}$ with respect to the usual inherited direct sum inner product. We denote by $I_{(X,v)}\in\mathcal{B}(\mathcal{H}_E,E)$ the isometry that equals to $I_E-vv^*$ on the $(X,v)$ slot and $0$ elsewhere.

\begin{corollary}\label{C:OperatorModel}
Let $(D(E))\ni 0$ and $F$ be as in Proposition~\ref{P:partialConcaveHypoGr} with $F|_D>0$. Fix a Hilbert space $E$ and an $\eta>0$. Assume that $\mathrm{co}^{\mathrm{mat}}(D)(E)$ has nonempty interior for $E$. Then there exists a vector $e\in\mathcal{H}_{E}$ with $\|e\|=1$, a completely bounded affine map $L_{F}:\mathcal{A}\otimes\mathcal{B}(E)\mapsto \mathcal{B}(\mathcal{H}_E)^*\otimes\mathcal{B}(E)$ given as
\begin{equation*}
L_{F}(X):=T_F\otimes I_E+\Lambda_{F}(X),
\end{equation*}
where $0\leq T_F\in\mathcal{B}(\mathcal{H}_E)^{*}$ and $\Lambda_{F}:\mathcal{A}\mapsto \mathcal{B}(\mathcal{H}_E)^*$ is a self-adjoint completely bounded linear map that is completely absolutely continuous with respect to $T_F$, such that
\begin{itemize}
\item[(a)]For all $X\in\mathrm{co}^{\mathrm{mat}}(D)(E)$ we have $L_{F}(X)\geq 0$;
\item[(b)]For all $(1+\eta)X\in D(E)$ in the interior of $\mathrm{co}^{\mathrm{mat}}(D)(E)$ and $v\in E$ we have
\begin{equation}\label{eq:C:OperatorModel}
\left\langle W,L_F(X)(I_{(X,v)}+ve^*)\right\rangle_{T_F}=e^*W^*F(X)v
\end{equation}
for all $W\in\overline{\mathcal{B}(\mathcal{H}_E,E)_{T_F}}$ with the notation of \eqref{eq:P:separating_pencil5}, and there exists an $\epsilon>0$ such that $\left\langle W,L_{F}(X)W\right\rangle_{T_F}\geq \epsilon T_F(W^*W)$.
\end{itemize}
\end{corollary}
\begin{proof}
Rewriting \eqref{eq:C:OperatorModel}, it essentially becomes
\begin{equation}\label{eq:C:OperatorModel1}
T(W^*(I_{(X,v)}+ve^*))+\lambda((I\otimes W^*)X(I\otimes (I_{(X,v)}+ve^*)))=e^*W^*F(X)v
\end{equation}
for $0\leq T\in\mathcal{B}(\mathcal{H}_E)^{*}$, $\lambda\in (\mathcal{A}\otimes \mathcal{B}(\mathcal{H}_E))^{*}$ and an $e\in\mathcal{H}_{E}$. If we equip $\mathcal{B}(\mathcal{H}_E)^{*}$, $(\mathcal{A}\otimes \mathcal{B}(\mathcal{H}_E))^{*}$ with their respective weak-$*$ topologies, 
then the set of all uniformly norm bounded $(T,\lambda)$ that satisfies \eqref{eq:C:OperatorModel1} for all $v,W,X$ and a fixed $e\in\mathcal{H}_{E}$ is closed in the product topology.
The free function $F$ preserves direct sums, so for any finite set of points $\{(X_1,v_1),\ldots,(X_n,v_n)\}$ with $v_i\in E_0$ and $(1+\eta)X_i\in D(E)$ in the interior of $\mathrm{co}^{\mathrm{mat}}(D)(E)$, by applying Proposition~\ref{P:separating_pencil} with $e=\oplus_{i=1}^n v_i$ to the single point data set $(\oplus_{i=1}^n X_i,\oplus_{i=1}^n v_i)$ it follows that the set of all such $(T,\lambda)$ is also nonempty. Moreover by (a) in Proposition~\ref{P:separating_pencil} we can assume that the norms of such $T,\lambda$ are uniformly bounded since $(1+\eta)X$ is also in the interior of $\mathrm{co}^{\mathrm{mat}}(D)(E)$. After a change of basis in $\mathcal{H}_{E}$ we conclude that we can also choose $e\in\mathcal{H}_{E}$ arbitrarily in \eqref{eq:C:OperatorModel1}. These norm bounded closed sets of $(T,\lambda)$ are compact it their respective weak-$*$ topologies by Banach-Alaoglu, thus their product is also compact. Furthermore these closed compact sets of $(T,\lambda)$ form a collection indexed by $(X,v)$ and thus can be partially ordered by inclusion within open sets of $(X,v)$ and then this collection has the finite intersection property. Thus, for a fixed $e\in\mathcal{H}_{E}$, by compactness there exists a $(T_F,\lambda_F)$ for which \eqref{eq:C:OperatorModel1} holds for all $(1+\eta)X\in D(E)$ in the interior of $\mathrm{co}^{\mathrm{mat}}(D)(E)$, $v\in E$ and $W\in\overline{\mathcal{B}(\mathcal{H}_E,E)_{T_F}}$. Then $L_{F}$ is determined by the transpose of $\lambda_F$, and the positivity condition $\left\langle W,L_{F}(X)W\right\rangle_{T_F}\geq \epsilon T_F(W^*W)$ follows from \eqref{eq:P:separating_pencil6}. Thus (a) and (b) are proved. 
\end{proof}

\begin{theorem}[Theorem 3 cf. \cite{anderson}]\label{T:Schur_complement}
Let $Z$ be a positive semi-definite linear operator on a Hilbert space and $S$ a subspace. Let the matrix of $Z$ be partitioned as $Z=\left[ \begin{array}{cc}
Z_{11} & Z_{12} \\
Z_{21} & Z_{22} \end{array} \right]$ with $Z_{11}:S\mapsto S$, $Z_{21}:S\mapsto S^{\perp}$. Then $\ran(Z_{21})\subset \ran(Z_{22})^{1/2}$ and there exists a bounded linear operator $C:S\mapsto S^{\perp}$ such that $Z_{21}=(Z_{22})^{1/2}C$ and
\begin{equation*}
Z=\left[ \begin{array}{cc}
Z_{11}-C^*C & 0 \\
0 & 0 \end{array} \right]+\left[ \begin{array}{cc}
C^* & 0 \\
(Z_{22})^{1/2} & 0 \end{array} \right]
\left[ \begin{array}{cc}
C & (Z_{22})^{1/2} \\
0 & 0 \end{array} \right].
\end{equation*}
The bounded positive semi-definite operator $\mathcal{S}_S(Z)=Z_{11}-C^*C$ is called the shorted operator or Schur complement of $Z$. It satisfies $\mathcal{S}_S(Z)\leq Z$ and it is maximal among all self-adjoint operators $X:S\mapsto S$ such that $X\leq Z$.
\end{theorem}

\begin{theorem}\label{T:reconstruct_schur}
Let $(D(E))\ni 0$ and $F$ be as in Proposition~\ref{P:partialConcaveHypoGr} with $F|_D>0$. Fix a Hilbert space $E$. Assume that $\mathrm{co}^{\mathrm{mat}}(D)(E)$ has nonempty interior for $E$. Then for each $X\in D(E)$ in the interior of $\mathrm{co}^{\mathrm{mat}}(D)(E)$ we have
\begin{equation}\label{eq:T:reconstruct_schur:0}
F(X)=(e\otimes I_E)\mathcal{S}_{e^*\otimes E}(L_{F}(X))(e^*\otimes I_E)
\end{equation}
where $L_{F}$ and $e$ are as in Corollary~\ref{C:OperatorModel} for an arbitrary, but sufficiently small fixed $\eta>0$. Moreover the right hand side of \eqref{eq:T:reconstruct_schur:0} is well defined for each interior point $X\in\mathrm{co}^{\mathrm{mat}}(D)(E)$.
\end{theorem}
\begin{proof}
By (b) in Corollary~\ref{C:OperatorModel} the self-adjoint completely bounded linear map $L_{F}(X)=T_F\otimes I_E+\Lambda_{F}(X)$ is strictly positive definite for all $X$ in the interior of $\mathrm{co}^{\mathrm{mat}}(D)(E)$ and is completely absolutely continuous with respect to $0\leq T_F\in\mathcal{B}(\mathcal{H}_E)^{*}$. Thus by Theorem~\ref{T:Schur_complement} its Schur complement pivoting on the subspace $e^*\otimes E$ of $\overline{\mathcal{B}(\mathcal{H}_E,E)_{T_F}}$ exists. By the strict positivity of $L_{F}(X)$, the $Z_{22}$ block of $L_{F}(X)$ in Theorem~\ref{T:Schur_complement} has closed range and is invertible. So, in \eqref{eq:C:OperatorModel} block Gaussian elimination applies and thus
\begin{equation*}
F(X)v=(e\otimes I_E)\mathcal{S}_{e^*\otimes E}(L_{F}(X))(e^*\otimes v)
\end{equation*}
for each $v\in E$ and $(1+\eta)X\in D(E)$ in the interior of $\mathrm{co}^{\mathrm{mat}}(D)(E)$. Taylor's power series formula ensures the uniqueness of analytic continuations to the whole interior of $\mathrm{co}^{\mathrm{mat}}(D)(E)$ thus the assertion holds for $\eta=0$ as well. This implies \eqref{eq:T:reconstruct_schur:0}.
\end{proof}

The converse of the above also holds:
\begin{theorem}\label{T:SchurConcave}
Let $\mathcal{H}$ a Hilbert space and $e\in\mathcal{H}$ with $\|e\|=1$ be fixed. Let a completely bounded affine map $L:\mathcal{A}\otimes\mathcal{B}(E)\mapsto \mathcal{B}(\mathcal{H})^*\otimes\mathcal{B}(E)$ be given as
\begin{equation*}
L(X):=T\otimes I_E+\Lambda(X),
\end{equation*}
where $0\leq T\in\mathcal{B}(\mathcal{H})^{*}$ and $\Lambda:\mathcal{A}\mapsto \mathcal{B}(\mathcal{H})^*$ is a self-adjoint completely bounded linear map that is completely absolutely continuous with respect to $T$. Then the function
\begin{equation}\label{eq:T:SchurConcave:0}
F(X)=(e\otimes I_E)\mathcal{S}_{e^*\otimes E}(L(X))(e^*\otimes I_E)
\end{equation}
is well defined and analytic for each $X\in\{Y\in\mathcal{A}\otimes\mathcal{B}(E):L(\Re(Y))>0\}$ and satisfies the assumptions of Proposition~\ref{P:partialConcaveHypoGr} with $D(E):=\{Y\in\mathcal{A}\otimes\mathcal{B}(E):L(Y)\geq 0\}$, in particular \eqref{eq:P:partialConcaveHypoGr} holds.
\end{theorem}
\begin{proof}
By the positivity assumption $L(\Re(X))>0$ the Schur complement in \eqref{eq:T:SchurConcave:0} is well defined and analytic as a free function in the sense of all various kinds of analyticities in \cite{verbovetskyi}. By Theorem~\ref{T:Schur_complement} the Schur complement of the positive affine map $L(X)$ is the maximal in the positive definite order on the subspace $e\otimes E$ among those which are dominated by $L(X)$ on the subspace $e\otimes E$ over its matrix convex domain $(D(E))$. From this maximality property concavity of $F(X)$ readily follows. Thus by Proposition~\ref{P:ConcaveIsom} $F(X)$ satisfies \eqref{eq:P:partialConcaveHypoGr}.
\end{proof}

Now as a combination of Theorem~\ref{T:reconstruct_schur} and Theorem~\ref{T:convex_hull_saturated} we obtain the following result establishing the analytic lifts for globally monotone functions on $\mathbb{P}(\mathbb{C})^k$. Then from this, further considerations prove the same for any other rectangular domain in $\mathbb{R}^k$. We use the notations $\Pi(E):=\{X\in\mathcal{B}(E):\Im(X)>0\}$ and $\overline{\Pi}(E):=\{X\in\mathcal{B}(E):\Im(X)\geq 0\}$, also $\Pi(E)^*:=\{X\in\mathcal{B}(E):\Im(X)<0\}$ and $\overline{\Pi}(E)^*:=\{X\in\mathcal{B}(E):\Im(X)\leq 0\}$.

\begin{theorem}\label{T:Loewner_several_var2}
Let $f:\mathbb{P}(\mathbb{C})^k\mapsto\mathbb{P}(\mathbb{C})$ be a real function. Then the following are equivalent:
\begin{itemize}
\item[(a)]$f$ is globally operator monotone;
\item[(b)]$f$ has a free analytic extension $f:\mathbb{P}(E)^k\mapsto \mathbb{P}(E)$ that is operator monotone;
\item[(c)]There exists a Hilbert space $\mathcal{K}$, a vector $e\in\mathcal{K}$, $0\leq B_i\in\mathcal{B}(\mathcal{K})$, $0\leq i\leq k$ with $B_0\geq \sum_{i=1}^kB_i$ such that for all $X\in\mathbb{CP}(E)^k$ we have
\begin{equation}\label{eq:T:Loewner_several_var2:0.1}
f(X)=(e\otimes I_E)\mathcal{S}_{e^*\otimes E}(L_{f}(X))(e^*\otimes I_E)
\end{equation}
where
\begin{equation}\label{eq:T:Loewner_several_var2:0.2}
L_{f}(X):=B_0\otimes I_E+\sum_{i=1}^kB_i\otimes (X_i-I_E),
\end{equation}
\item[(d)]$f$ has a free analytic continuation to $\Pi(E)^k$ and to $(\Pi(E)^*)^k$ across $\mathbb{P}(E)^k$, mapping $\Pi(E)^k$ to $\overline{\Pi}(E)$ and $(\Pi(E)^*)^k$ to $\overline{\Pi}(E)^*$.
\end{itemize}
\end{theorem}
\begin{proof}
First we prove that (a) implies (c). By Theorem~\ref{T:reconstruct_schur} the representation formula follows for the translated function $g(x):=f(x+1)$ with domain $(-1,\infty)^k$ whose matrix convex hull is $D(E):=\{X\in\mathcal{B}(E):X=X^*,X_i\geq -I_E\}^k$ which contains an open neighborhood of $0$ for the operator system $\mathcal{A}=\mathbb{C}^k$. Since the domain contains arbitrarily large positive operators and by (a) of Corollary~\ref{C:OperatorModel} we have $L_{f}(X)\geq 0$, it follows that $\Lambda_{f}(X)$ is completely positive, thus of the the form as in \eqref{eq:T:Loewner_several_var2:0.2} with $B_i\geq 0$ and $B_0\geq \sum_{i=1}^kB_i$ as well, since $L_{f}(t I)\geq 0$ for all $t>0$. In a similar way we show that (b) implies (c).

Next we claim that (c) implies (a). Indeed, $L_{f}(X)$ is order preserving and the maximality characterization of the Schur complement in Theorem~\ref{T:Schur_complement} ensures that the right hand side of \eqref{eq:T:Loewner_several_var2:0.1} is an operator monotone function. In a similar way we also prove that (c) implies (b).

That (c) implies (d) essentially follows from $B_i\geq 0$ so that the Schur complement in \eqref{eq:T:Loewner_several_var2:0.1}, if it exists, has strictly positive imaginary part if $\Re(X_i)>0$, see Proposition 5.1. \cite{palfia1}. Now the strict positivity of $\Im(L_{f}(X))$ for arbitrary $\Im(X)\in\Pi(E)^k$ and $E$ implies the strict lower boundedness, thus the existence of the inverse operator in the formula of the Schur complement, thus the Schur complement itself as an analytic function on $\Pi^k$.

Lastly, that (d) implies (b) can be found in the main theorem of \cite{palfia1}.
\end{proof}

At this point, much like as in \cite{palfia1}, one can use M\"obius transformations to transform the domain $\mathbb{P}(\mathbb{C})^k$ into any open rectangle in $\mathbb{R}^k$ to prove that global monotonicity implies free analytic continuation for the function to the whole matrix convex hull of its domain, thus arriving at its operator monotone non-commutative lift.

\begin{remark}
In the above corollaries we assumed $\mathcal{A}=\mathbb{C}^k$, however one can work out similar results in the same way for any operator system $\mathcal{A}$.
\end{remark}

\section{Representation of operator means of probability measures}



Let $\mathcal{P}(\mathbb{P}(E))$ denote the set of fully supported Borel probability measures on the complete metric space $(\mathbb{P}(E),d_\infty)$ where $E$ is a Hilbert space and $d_\infty(A,B)=\|\log(A^{-1/2}BA^{-1/2})\|$ denotes the Thompson metric \cite{hiailim,palfia2}. Let $\mathcal{P}^\infty(\mathbb{P}(E))\subset\mathcal{P}(\mathbb{P}(E))$ denote the subset of probability measures with bounded support. For a $\mu\in\mathcal{P}(\mathbb{P}(E))$ the support $\supp(\mu)$ is a separable closed subset of $\mathbb{P}(E)$ and it has full measure $\mu(\supp(\mu))=1$. Also note that the relative operator norm topology on 
$\mathbb{P}(E)$ agrees with the metric topology of $d_\infty$, for this and further references see \cite{palfia2}.

\begin{proposition}\label{P:embedP}
The collection of sets $(\mathcal{P}^\infty(\mathbb{P}(E)))$ indexed by $E$ is a self-adjoint matrix convex set. In particular $\mathcal{P}^\infty(\mathbb{P}(E))$ embeds into $L^\infty([0,1],\lambda)_+\otimes \mathbb{P}(E)$, the strictly positive cone of the projective tensor product $L^\infty([0,1],\lambda)\otimes \mathcal{B}(E)$.
\end{proposition}
\begin{proof}
Let $\mu\in\mathcal{P}^\infty(\mathbb{P}(E))$. Then $\supp(\mu)$ is separable and closed, thus $\mu$ is concentrated on the complete Polish space $(\supp(\mu),d_\infty)$. Thus by the Skorokhod representation Theorem 8.5.4. in \cite{bogachev}, there exists a Borel map $\xi_\mu:[0,1]\mapsto \supp(\mu)\subseteq\mathbb{P}(E)$ such that $\mu=(\xi_\mu)_*\lambda$. Since $\supp(\mu)$ is a bounded subset of $\mathbb{P}(E)$, we have that $\esssup \|\xi_\mu\|<\infty$. Thus $\xi_\mu\in L^\infty([0,1],\lambda,\mathbb{P}(E))\subseteq L^\infty([0,1],\lambda,\mathcal{B}(E))$ where 
\begin{equation*}
\begin{split}
L^\infty([0,1],\lambda,\mathcal{B}(E))=\{&f:[0,1]\mapsto\mathcal{B}(E), f \text{ is strongly measurable},\\
&\esssup(f)<\infty\}.
\end{split}
\end{equation*}
Next, we claim that $L^\infty([0,1],\lambda,\mathcal{B}(E))\simeq L^\infty([0,1],\lambda)\otimes \mathcal{B}(E)$ as von Neumann algebras, where the latter is the projective tensor product. This follows from the same argument leading to Proposition 12.5. in \cite{paulsen} showing that $C(X)\otimes \mathcal{A}\simeq C(X,\mathcal{A})$ for any $C^*$-cross norm with the isomorphism $\phi:\sum_{i=1}^nf_i\otimes a_i\longrightarrow \sum_{i=1}^nf_i(t)a_i$ for $f_i\in C(X)$ and $a_i\in\mathcal{A}$, where $\mathcal{A}$ is a $C^*$-algebra and $X$ is a compact Hausdorff space. Now it is straightforward to see that the Borel map constructed above $\xi_\mu$ is a strictly positive element of the positive cone of $L^\infty([0,1],\lambda)\otimes \mathcal{B}(E)$ which is $L^\infty([0,1],\lambda)^+\otimes \mathbb{P}(E)$ by Lemma 2.2. in \cite{stormer}, where $L^\infty([0,1],\lambda)^+=\{f\geq 0:f\in L^\infty([0,1],\lambda)\}$. 

Now the positive cone $(L^\infty([0,1],\lambda)^+\otimes \mathbb{P}(E))$ is closed under direct sums and isometric conjugations, thus $(L^\infty([0,1],\lambda)^+\otimes \mathbb{P}(E))$ is an (open) matrix convex set. Moreover for any $\xi\in L^\infty([0,1],\lambda)^+\otimes \mathbb{P}(E)$ the pushforward $(\xi)_*\lambda\in\mathcal{P}^\infty(\mathbb{P}(E))$, so $(\mathcal{P}^\infty(\mathbb{P}(E)))$ is matrix convex as well.
\end{proof}

\begin{remark}\label{R:EquivalentCrossNorms}
Notice that $L^\infty([0,1],\lambda)$ is an injective von Neumann algebra, or in other words $L^\infty([0,1],\lambda)$ is nuclear as a $C^*$-algebra. Nuclearity ensures that all $C^*$-cross norms on $(L^\infty([0,1],\lambda)^+\otimes \mathbb{P}(E))$ are equivalent. So in particular the projective and injective $C^*$-cross norms are the same. For more details see for example Chapter IV in \cite{takesaki}.
\end{remark}

\begin{remark}\label{R:ForgetsProduct}
In Proposition~\ref{P:embedP} the embedding of $(\mathcal{P}^\infty(\mathbb{P}(E)))$ into the cone $L^\infty([0,1],\lambda)_+\otimes \mathbb{P}(E)$ does not appear to be injective. Elements of $L^\infty([0,1],\lambda)_+\otimes \mathbb{P}(E)$ can be classified into equivalence classes by almost sure identification with elements of $(\mathcal{P}^\infty(\mathbb{P}(E)))$.
\end{remark}

A set $U\subseteq \mathbb{P}(E)$ is \emph{upper} if $X\leq Y\in \mathbb{P}(E)$ and $X\in U$ imply that $Y\in U$.
\begin{definition}[Stochastic order, cf. \cite{hiailawsonlim}]
For $\mu,\nu\in\mathcal{P}(\mathbb{P}(E))$ the stochastic partial order $\mu\leq\nu$ is defined by requiring $\mu(U)\leq\nu(U)$ for all closed upper sets $U\subseteq\mathbb{P}(E)$.
\end{definition}

The following result is essentially due to Strassen for Polish spaces with a closed partial order. It can be found in a suitable form as Theorem 1 in \cite{kamaekrengelobrien}.

\begin{theorem}\label{T:StrassenMonotone}
Let $\mu,\nu\in\mathcal{P}^\infty(\mathbb{P}(E))$. Then the following are equivalent:
\begin{itemize}
\item[(i)]$\mu\leq\nu$;
\item[(ii)]there exists $\xi_\mu:[0,1]\mapsto \supp(\mu)$ and $\xi_\nu:[0,1]\mapsto \supp(\nu)$ such that $\mu=(\xi_\mu)_*\lambda$ and $\nu=(\xi_\nu)_*\lambda$ with $\xi_\mu(t)\leq\xi_\nu(t)$ almost surely for all $t\in[0,1]$.
\end{itemize}
\end{theorem}



One might wonder for $\mu\in\mathcal{P}^\infty(\mathbb{P}(E))$, $\nu\in\mathcal{P}^\infty(\mathbb{P}(K))$ what is the correct way to define $\mu\oplus\nu$? In \cite{hiailim} the authors define it as the pushforward $(g)_*(\mu\times\nu)$ of the direct sum $g(A,B):=A\oplus B$, and they show that their operator means preserve this direct sum. However in free function theory if we have $n$-tuples of operators $X:=(X_1,\ldots,X_n)\in\mathcal{B}(E)^n$ and $Y:=(Y_1,\ldots,Y_n)\in\mathcal{B}(K)^n$, their direct sum is element-wise, that is $X\oplus Y=(X_1\oplus Y_1,\ldots,X_n\oplus Y_n)$. Free functions, including as well all operator means \cite{palfia1}, preserve this direct sum. Notice that if we regard $X,Y$ as discrete probability measures, that is $X=\sum_{i=1}^n\frac{1}{n}\delta_{X_i}$, $Y=\sum_{i=1}^n\frac{1}{n}\delta_{Y_i}$, then both definitions of $X\oplus Y$ are measures in $\mathcal{P}(\mathcal{B}(E\oplus K))$ with marginals $\mu$ and $\nu$, where the $\sigma$-algebra is induced by the norm. Also operator means are permutation invariant, that is the ordering of coordinates in $(X_1,\ldots,X_n)$ does not matter \cite{hiailim,palfia1}. This and Remark~\ref{R:ForgetsProduct} seem to suggest that we should allow some non-uniqueness when considering direct sums of measures. This leads to the following.

\begin{definition}[Direct sums of probability measures]\label{D:directsumProbMeas}
For $\mu\in\mathcal{P}^\infty(\mathbb{P}(E))$, $\nu\in\mathcal{P}^\infty(\mathbb{P}(K))$, let $\Gamma(\mu,\nu)\subseteq\mathcal{P}^\infty(\mathbb{P}(E\oplus K))$ denote the set of couplings of $\mu,\nu$, that is $\gamma\in\Gamma(\mu,\nu)$ if $\gamma(A\times \mathbb{P}(K))=\mu(A)$ and $\gamma(\mathbb{P}(E)\times B)=\nu(B)$, in other words, elements of $\Gamma(\mu,\nu)$ have marginals $\mu,\nu$. Then $\mu\oplus\nu$ is defined to be the set $\Gamma(\mu,\nu)$. Thus in general, the direct sum of probability measures is no longer uniquely determined.
\end{definition}

Notice that $\Gamma(\mu,\nu)$ is nonempty, since the product measure $\mu\times\nu\in\Gamma(\mu,\nu)$. Also for any $\gamma\in\Gamma(\mu,\nu)$ we have that $\supp(\gamma)\subseteq \supp(\mu)\times\supp(\nu)$. We may regard operator means of finitely supported measures as a sequence of functions satisfying the following.

\begin{definition}[Operator mean of discrete probability measures]
For each $0<n\in\mathbb{N}$ and Hilbert space $E$ let $F_n:\mathbb{P}(E)^n\mapsto\mathbb{P}(E)$ be an operator monotone free function. Then we say that the sequence of functions $F_n$ is an operator mean if it satisfies the following
\begin{itemize}
\item[1)]For a permutation $\sigma\in S_n$, $F_n(X_1,\ldots,X_n)=F_n(X_{\sigma(1)},\ldots,X_{\sigma(n)})$;
\item[2)]For $0<k\in\mathbb{N}$, $F_{nk}(\underbrace{X_1,\ldots,X_1}_{\text{$k$ times}},\ldots,\underbrace{X_n,\ldots,X_n}_{\text{$k$ times}})=F_n(X_1,\ldots,X_n)$.
\end{itemize}
In order to simplify notation, the subscript $n$ will often be omitted, simply writing $F=F_n$ for an operator mean.
\end{definition}

Notice that given an operator mean $F=F_n$, it is automatically defined for discrete probability measures with rational weights by grouping together the repeated variables and applying 1), 2). It is known that operator concavity for a free function $F:\mathbb{P}(E)^n\mapsto\mathbb{P}(E)$ is equivalent to operator monotonicity which in turn is equivalent to \eqref{eq:P:partialConcaveHypoGr} by an argument similar to the one in Proposition~\ref{P:monotone_concave}, for more details see \cite{palfia1}.

\begin{proposition}\label{P:opmeanDirectsum}
An operator mean $F_n:\mathbb{P}(E)^n\mapsto\mathbb{P}(E)$ preserves direct sums of discrete probability measures with rational weights in the sense of Definition~\ref{D:directsumProbMeas}.
\end{proposition}
\begin{proof}
Without loss of generality, let $\mu=\sum_{i=1}^n\frac{1}{n}\delta_{X_i}$, $\nu=\sum_{i=1}^k\frac{1}{k}\delta_{Y_i}$ be given. Then any $\mu\oplus\nu$ is supported on the set $\supp(\mu)\times\supp(\nu)=\{X_i\oplus Y_j:i\in\{1,\ldots,n\},j\in\{1,\ldots,k\}\}$. Then using the direct sum invariance of $F$ and grouping elements by 2), we obtain that $F(\mu\oplus\nu)=F(\mu)\oplus F(\nu)$.
\end{proof}

In order to study operator means of general probability measures, instead of considering the restrictive set of functions $F:\mathcal{P}^\infty(\mathbb{P}(E))\mapsto\mathbb{P}(E)$ we consider first free functions of random variables, that is $F:(L^1([0,1],\lambda)^+\otimes \mathbb{P}(E))\mapsto\mathbb{P}(E)$. Let $S([0,1],\lambda)$ denote the set of simple functions on $[0,1]$. Then $S([0,1],\lambda)$ is norm-dense in $L^p([0,1],\lambda)$ for $1\leq p\leq +\infty$ and the same is true for $S([0,1],\lambda)^+\otimes \mathbb{P}(E)$ in $L^p([0,1],\lambda)^+\otimes \mathbb{P}(E)$.

\begin{theorem}\label{T:OpMeanRandomVar}
Assume that $F:S([0,1],\lambda)^+\otimes \mathbb{P}(E)\mapsto\mathbb{P}(E)$ is free function that satisfies \eqref{eq:P:partialConcaveHypoGr}. Then for each $1\leq p\leq +\infty$ there exists a unique $\hat{F}_p:L^p([0,1],\lambda)^+\otimes \mathbb{P}(E)\mapsto\mathbb{P}(E)$ extending $F$.
\end{theorem}
\begin{proof}
In essence $F$ can be regarded as a sequence of free functions indexed by $0<n\in\mathbb{N}$ for each Hilbert space $E$, that is $F_n:\mathbb{P}(E)^n\mapsto\mathbb{P}(E)$ satisfying the assumptions of Proposition~\ref{P:partialConcaveHypoGr}. Then we can apply Corollary~\ref{C:OperatorModel} so we have \eqref{eq:C:OperatorModel} for each $F_n$. Each $F_n$ is thus norm-continuous by Theorem~\ref{T:SchurConcave}, so $F:S([0,1],\lambda)^+\otimes \mathbb{P}(E)\mapsto\mathbb{P}(E)$ is relative norm-continuous with respect to the norm topology of $L^p([0,1],\lambda)^+\otimes \mathbb{P}(E)$. Then, it admits a unique norm-continuous extension $\hat{F}_p:L^p([0,1],\lambda)^+\otimes \mathbb{P}(E)\mapsto\mathbb{P}(E)$, since $S([0,1],\lambda)^+\otimes \mathbb{P}(E)$ is norm dense in $L^p([0,1],\lambda)^+\otimes \mathbb{P}(E)$.
\end{proof}

\begin{proposition}\label{P:ConcaveMonotone}
Let $F:(S([0,1],\lambda)^+\otimes \mathbb{P}(E))\mapsto\mathbb{P}(E)$ be an operator monotone free function. Then $F$ satisfies \eqref{eq:P:partialConcaveHypoGr}.
\end{proposition}
\begin{proof}
The same argument as in the proof of Theorem~\ref{T:convex_hull_saturated} applies for $f:=F:(S([0,1],\lambda)^+\otimes \mathbb{P}(E))\mapsto\mathbb{P}(E)$ where $S([0,1],\lambda)^+\otimes \mathbb{P}(E)$ is substituted for both $\mathbb{P}(\mathbb{C})^k$ and $\mathbb{CP}(E)^k$, so we get that for each $(Y,X)\in\mathrm{co}^{\mathrm{mat}}(\hypo(f))(E)$ we have that $Y\leq f(X)$. Note that in Theorem~\ref{T:convex_hull_saturated} $\dim(E)<+\infty$ is assumed, but actually does not affect its proof. The claim that for each $(Y,X)\in\mathrm{co}^{\mathrm{mat}}(\hypo(f))(E)$ we have that $Y\leq f(X)$ combined with the implication '$\Leftarrow$' of Proposition~\ref{P:partialConcaveHypoGr} proves that $F$ satisfies \eqref{eq:P:partialConcaveHypoGr}. A similar argument to this can also be found in \cite{palfia1}.
\end{proof}

By the density of $S([0,1],\lambda)^+$ in $L^1([0,1],\lambda)^+$ we immediately obtain:

\begin{corollary}\label{C:ConcaveMonotone}
Let $F:(L^1([0,1],\lambda)^+\otimes \mathbb{P}(E))\mapsto\mathbb{P}(E)$ be an operator monotone free function. Then $F$ satisfies \eqref{eq:P:partialConcaveHypoGr}.
\end{corollary}

\begin{theorem}\label{T:OpMeanProbMeas}
Assume that the sequence of functions $F_n:\mathbb{P}(E)^n\mapsto\mathbb{P}(E)$ for $0<n\in\mathbb{N}$ is an operator mean of discrete probability measures. Then it uniquely extends into a stochastic order preserving function $\hat{F}:\mathcal{P}^\infty(\mathbb{P}(E))\mapsto\mathbb{P}(E)$.
\end{theorem}
\begin{proof}
Pairs of probability measures admit Skorokhod representations that are order preserving by Theorem~\ref{T:StrassenMonotone}. Then through the Skorokhod representation we obtain a lift $\hat{F}:S([0,1],\lambda)^+\otimes \mathbb{P}(E)\mapsto\mathbb{P}(E)$ representing the sequence of functions $F_n:\mathbb{P}(E)^n\mapsto\mathbb{P}(E)$, such that $\hat{F}$ is operator monotone. Then by Proposition~\ref{P:ConcaveMonotone} $\hat{F}$ satisfies \eqref{eq:P:partialConcaveHypoGr}, so existence and uniqueness of the extension $\hat{F}:L^\infty([0,1],\lambda)^+\otimes \mathbb{P}(E)\mapsto\mathbb{P}(E)$ follows from Theorem~\ref{T:OpMeanRandomVar}. Then Corollary~\ref{C:OperatorModel} applies to the translated function $G(X):=\hat{F}(X+I)$ where $I(t):=I_E$ for all $t\in[0,1]$. By (a) of Corollary~\ref{C:OperatorModel} we have $L_{G}(X)\geq 0$ for any $(X+I)\in L^\infty([0,1],\lambda)^+\otimes \mathbb{P}(E)$, so it follows that $0\leq\lambda$ in \eqref{eq:C:OperatorModel1}, thus $L_{G}(X)$ is order preserving. By the maximal characterization of the Schur complement in Theorem~\ref{T:Schur_complement} it follows that it is also order preserving. Thus $\hat{F}$ is also order preserving, i.e. operator monotone. Then its restriction $\hat{F}:\mathcal{P}^\infty(\mathbb{P}(E))\mapsto\mathbb{P}(E)$ preserves the stochastic order.
\end{proof}

\begin{corollary}\label{C:OpMeansProbMeasure}
Let $F:\mathcal{P}^\infty(\mathbb{P}(E))\mapsto\mathbb{P}(E)$ be a stochastic order preserving free function. Then there exists an operator monotone free function $\hat{F}:L^\infty([0,1],\lambda)^+\otimes \mathbb{P}(E)\mapsto\mathbb{P}(E)$ that represents $F$ and $\hat{F}(X+I)$ is of the form \eqref{eq:C:OperatorModel1} where $0\leq\lambda\in (L^\infty([0,1],\lambda)\otimes \mathcal{B}(\mathcal{H}_E))^{*}$ and $I(t):=I_E$ for all $t\in[0,1]$. In particular $\hat{F}(X+I)$ is given by \eqref{eq:T:reconstruct_schur:0}.
\end{corollary}
\begin{proof}
Through the Skorokhod representation Theorem~\ref{T:StrassenMonotone} we obtain the lift $\hat{F}:L^\infty([0,1],\lambda)^+\otimes \mathbb{P}(E)\mapsto\mathbb{P}(E)$ which by Corollary~\ref{C:ConcaveMonotone} satisfies \eqref{eq:P:partialConcaveHypoGr}. Then Corollary~\ref{C:OperatorModel} applies to the translated function $G(X):=\hat{F}(X+I)$. By (a) of Corollary~\ref{C:OperatorModel} $L_{G}(X)\geq 0$ for any $(X+I)\in L^\infty([0,1],\lambda)^+\otimes \mathbb{P}(E)$, it follows that $0\leq\lambda$ in \eqref{eq:C:OperatorModel1} and \eqref{eq:T:reconstruct_schur:0} holds.
\end{proof}

\section*{Acknowledgment}
This work was partly supported by the grant TUDFO/47138-1/2019-ITM of the Ministry for Innovation and Technology, Hungary; the Hungarian National Research, Development and Innovation Office NKFIH, Grant No. FK128972; and the National Research Foundation of Korea (NRF) grant funded by the Korea government (MEST) No.2015R1A3A2031159, No.2016R1C1B1011972 and No.2019R1C1C1006405.




\end{document}